\documentclass{amsart}
\usepackage[utf8]{inputenc}
\usepackage{amsmath}
\usepackage{amsthm}
\usepackage{amsfonts}
\usepackage[margin=1.2in]{geometry}
\usepackage{amscd}
\usepackage{xcolor}
\usepackage{float, pifont}
\usepackage{tabularx}
\usepackage{graphicx}
\usepackage{enumitem}
\usepackage{amssymb}
\usepackage{hyperref}
\usepackage{ytableau}
\usepackage{tikz}
\usetikzlibrary{positioning,calc,matrix,shapes.multipart}
\usepackage{dsfont}
\usepackage{framed}
\usepackage{afterpage}
\usepackage{caption}
\usepackage{subcaption}
\usepackage{epsfig}
\usepackage{wrapfig}

\usetikzlibrary{3d}
\usetikzlibrary{positioning,matrix,arrows,decorations.pathmorphing}

\tikzstyle{invisivertex} = [black, shape=circle, minimum size=0pt, inner sep=0pt]

\newtheorem{theorem}{Theorem}
\newtheorem{lemma}[theorem]{Lemma}

\newtheorem{cor}[theorem]{Corollary}
\newtheorem{proposition}[theorem]{Proposition}

\theoremstyle{remark}
\newtheorem{rem}[theorem]{Remark}

\theoremstyle{definition}

\newtheorem{example}[theorem]{Example}

\makeatletter
\newtheorem*{rep@theorem}{\rep@title}
\newcommand{\newreptheorem}[2]{%
\newenvironment{rep#1}[1]{%
 \def\rep@title{#2 \ref{##1}}%
 \begin{rep@theorem}}%
 {\end{rep@theorem}}}
\makeatother

\newreptheorem{theorem}{Theorem}

\DeclareMathOperator{\dist}{dist}
\DeclareMathOperator{\az}{Az}
\DeclareMathOperator{\Z}{\mathbb{Z}}

\DeclareMathOperator{\K}{{\sf K}}
\DeclareMathOperator{\R}{{\sf R}}

\DeclareMathOperator{\PPP}{{\mathcal{P}}}

\newcommand{\weight}{\mathrm{weight}}
\newcommand{\len}{\mathrm{len}}

\title[A Three-regime Theorem for flow-firing]{A Three-regime Theorem for flow-firing}

\author[Brauner]{Sarah Brauner}
\email{sarah\_brauner@brown.edu}
\address{Brown University, Providence, RI, USA}

\author[Dorpalen-Barry]{Galen Dorpalen-Barry}
\email{dorpalen.barry@tamu.edu}
\address{Texas A\&M University, College Station, TX, USA}

\author[Kara]{Selvi Kara}
\email{skara@brynmawr.edu}
\address{Bryn Mawr College, Bryn Mawr, PA, USA}

\author[Klivans]{Caroline Klivans}
\email{caroline\_klivans@brown.edu}
\address{Brown University, Providence, RI, USA}

\author[Schneider]{Lisa Schneider}

\begin{document}

\begin{abstract}
    Graphical chip-firing is a discrete dynamical system where chips are placed on the vertices of a graph and exchanged via simple firing moves. Recent work has sought to generalize chip-firing on graphs to higher dimensions, wherein graphs are replaced by cellular complexes and chip-firing becomes flow-rerouting along the faces of the complex. Given such a system, it is natural to ask (1) whether this firing process terminates and (2) if it terminates uniquely (i.e. is confluent). In the graphical case, these questions were definitively answered by Bjorner--Lovasz--Shor, who developed three regimes which completely determine if a given system will terminate. Building on the work of Duval--Klivans--Martin and Felzenszwalb-Klivans, we answer these questions in a context called flow-firing, where the cellular complexes are 2-dimensional.
\end{abstract}
\maketitle
\section{Introduction}
This paper concerns {\bf flow-firing}, a higher-dimensional analogue of chip-firing.
In classical chip-firing  chips are placed on the vertices of a  graph and move to other vertices via local moves dictated by the graph Laplacian~\cite{klivans-book}. Flow-firing is a discrete model for cell complexes, where flow along one-dimensional edges is diverted over two-dimensional faces as dictated by the combinatorial Laplacian.  Higher dimensional chip-firing was introduced in \cite{criticalsimplicial} where the focus was on algebraic considerations. The study of the dynamical properties of flow-firing was initiated in \cite{flowfiringprocesses}, and is continued here. Consider the following two questions 
 concerning the fundamental behavior of  termination and confluence for flow-firing:

$\bullet$ For which initial configurations does the flow-firing process terminate?

$\bullet$ For which initial configurations does the flow-firing process terminate uniquely?

In the classical case of chip-firing on a graph the answer to the first question is provided by a theorem of Bj\"{o}rner, Lovasz and  Shor~\cite{bjorner-lovasz-shor}.   Their result identifies three regimes for chip-firing behavior.  Informally, the theorem states that: (1) If the number of chips is small enough then the process always terminates; (2) If the number of chips is large enough then the process never terminates; and
(3) If the number of chips is in a middle range, then one can always find a configuration which terminates and one that does not, see Theorem~\ref{BLS} for a precise statement.  

In the graphical case, if a configuration terminates it always does so uniquely, rendering the second question unnecessary.  In flow-firing, however, not all initial configurations terminate uniquely.  An important setting that terminates uniquely was identified in \cite{flowfiringprocesses}. Furthermore the authors conjectured a related much larger class of initial configurations would terminate uniquely.   Building from that work, our main result reveals a subtler behaviour than originally conjectured, resulting in a generalization of 
the Three-Regime Theorem for uniqueness of termination for flow-firing.  

Denote by $\K(n,r)$ the initial configuration known as the \emph{pulse} of height $n$ and radius $r$ (see Section~\ref{sec:background}).

\begin{theorem}[Three-Regime Theorem for Flow-Firing]\label{thm:3regimes}
Let $r,n\in\Z_{\geq 0}$.
\begin{enumerate}
\item\label{3regime1} If $r \leq 1$, then firing from $\K(n,r)$ terminates uniquely in the Aztec diamond.
\item\label{3regime2} If $2 \leq r\leq \lceil\frac{n}{2}\rceil$, then firing from $\K(n,r)$
\begin{enumerate}
    \item\label{3regime2-a} does not terminate uniquely, but
    \item\label{3regime2-b} can terminate in the Aztec diamond.
\end{enumerate}
\item\label{3regime3} If $r > \lceil\frac{n}{2}\rceil$, then
\begin{enumerate}
    \item\label{3regime3-a} $\K(n,r)$ does not terminate uniquely, and
    \item\label{3regime3-b} if $r \geq \lceil \frac{n}{\sqrt{3}} \rceil +1$, then firing from $\K(n,r)$ will never terminate in the Aztec diamond.
\end{enumerate}
\end{enumerate}
\end{theorem}
Questions of termination and uniqueness have also been considered in other variations of chip-firing.  Two notable examples are labeled chip-firing~\cite{hopkins-mcconville-propp, ckpl, ckpl2} and root-system chip-firing~\cite{roots1, roots2} where the emphasis is on local versus global confluence properties.  
We  prove confluence (unique termination) for subsystems of flow-firing including certain path-firings (Corollary \ref{lem:every-face-fires}).

The remainder of the paper proceeds as follows. Section \ref{sec:background} gives background on flow-firing. We introduce the notion of firing along a path and study its behavior in Section \ref{sec:pathFire}. In Sections \ref{sec:regimes}, \ref{sec:regime2} and \ref{sec:regime_3} we prove Theorem \ref{thm:3regimes}, addressing each of the three regimes in turn. We then conclude in Section \ref{sec:closing} by discussing challenges to improving the bounds of Theorem \ref{thm:3regimes}.

\subsection*{Acknowledgements}
This project was initiated as part of the ICERM Research Community in Algebraic Combinatorics in 2021/2022. We are thankful both to ICERM and the organizers (Pamela E. Harris, Susanna Fishel, Rosa Orellana, and Stephanie van Willigenburg) for their generous support. The first author was supported by an NSF Graduate Research Fellowship and NSF MSPRF DMS-2303060. All authors thank the anonymous reviewer for pointing out Proposition \ref{prop:discretederivative}, i.e. that the setup in Section 3 is equivalent to chip-firing on a line via the discrete derivative. 

\section{Background}\label{sec:background}

Following \cite{flowfiringprocesses}, we consider flow-firing on the two dimensional grid complex.  A {\bf flow configuration} is an assignment of integer values to the edges of the complex.  Visually, we depict flow as oriented, with positive values oriented to the right and up, and negative values oriented to the left and down.  For example, the flow configuration on the left in Figure~\ref{fig1} has flow value zero on all edges except the center edge, which has a value of $2$ and is depicted as a downward facing arrow and labeled with the magnitude $2$. 

The {\bf degree} of an edge is the number of two dimensional faces in which it is contained.  
An edge can {\bf fire} if it has at least as many units of flow as its degree.  In firing an edge, one unit of flow is diverted around each of the two dimensional cells containing it; an example is shown in Figure~\ref{fig1}.  In visualizations, an edge with $1$ unit of flow will have a directed arrow but we suppress the number $1$. 

    \begin{figure}[H]
    \centering
    \includegraphics[width=0.5\textwidth]{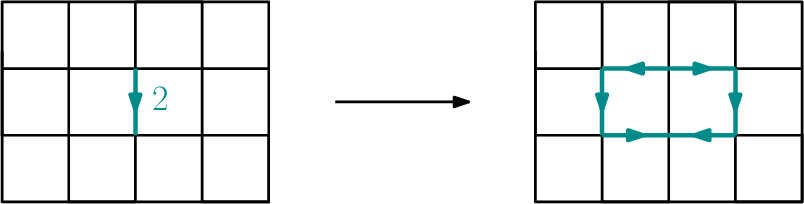}
    \caption{An example of flow-firing on the grid}
    \label{fig1}
\end{figure}

The flow-firing process starts with a fixed initial flow configuration. From all edges that can fire, one is selected and fired.  This continues until no firing moves are possible or continues forever.  A flow configuration is {\bf stable} if no flow-firing moves are possible. 
From a fixed initial configuration, we are interested in whether a stable configuration can be reached or not, and whether that stable configuration is unique. 
A  flow configuration is {\bf conservative} if for each vertex $v$, the flow into $v$ is equal to the flow out of $v$. 

\begin{theorem}[{{\cite[Proposition 2, Theorem 4]{flowfiringprocesses}}}]
Let $\K$ be a flow configuration.
\begin{itemize}
    \item If there is a vertex $v$ with
    \[
    | {\sf inflow}( v)  - {\sf outflow}( v)| > 4,
    \]
    then flow-firing never terminates.
    \item If $\K$ is conservative, then flow-firing terminates after a finite number of steps.
\end{itemize}
\end{theorem}

Moreover, conservative flow configurations are induced by {\bf face representations}.  A face representation of a flow configuration is an assignment of integer values to the $2$-cells of the complex.  Visually, flow is depicted as oriented, with positive values oriented clockwise and negative values oriented counter-clockwise, see below.  

\begin{figure}[H]
\centering
\begin{subfigure}{.4\textwidth}
  \centering
  \includegraphics[width=.45\linewidth]{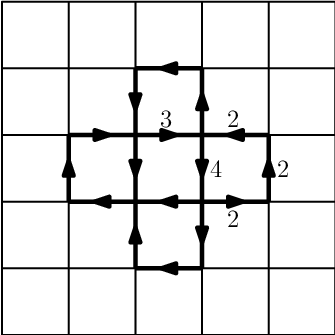}
  \caption{Edge representation}
  \label{fig:sub1}
\end{subfigure}
\begin{subfigure}{.4\textwidth}
  \centering
  \includegraphics[width=.45\linewidth]{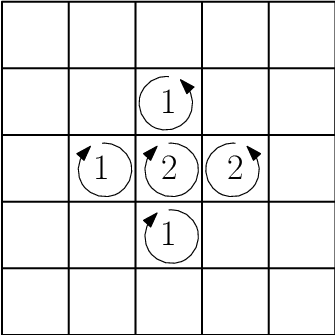}
  \caption{Face representation}
  \label{fig:sub2}
\end{subfigure}
\label{fig:test}
\end{figure}

Henceforth, we assume our configurations are conservative and can therefore be analyzed using the face representation. Given such a configuration $\K$, we denote the weight of a face $f$ by $\K_f$. The {\bf support} of a face configuration is the set of faces $f$ whose weights are nonzero. In the grid, the distance between faces $f$ and $g$, written $\dist(f, g)$ is the \emph{Manhattan distance}. For instance, $\dist(f, f) = 0$, and if $f$ and $g$ are neighbors, i.e. share an edge, then $\dist(f,g) = 1$.

The flow-firing rules described earlier can be translated to firing rules on faces. Let $\K$ be a configuration and suppose $f,g$ are faces of the grid with weights $\K_f$ and $\K_g$ respectively.
If $\K_f \geq \K_g + 2$, then a {\bf face fire} of $f$ towards $g$ results in a new configuration $\K'$ with weights $\K'_f = \K_f - 1$ and  $\K'_g = \K_g + 1$.  If the recipient $g$ is understood from context we will often simply refer to firing a face $f$.

\begin{example}\label{ex: flow-firing}
Below are three face configurations. The second picture (b) is obtained from the first picture (a) by firing the face with weight $2$ to the right. The last picture (c) is obtained from (a) by by firing the face with weight $3$ to the left.
\[
\begin{ytableau}
\textcolor{white}{.}    &  3  & \\
 1 & 2  & \\
     &   &     \\
\end{ytableau}
\hspace{2cm}
\begin{ytableau}
\textcolor{white}{.}
   &  3   & \\
 1 &  1 & 1    \\
     &   &     \\
\end{ytableau}
\hspace{2cm}
\begin{ytableau}
\textcolor{white}{.}
  1 &  2   & \\
  1 &  2   &  \\
      &   &     \\
\end{ytableau}
\]
\begin{center}
(a) \hspace{3cm}
(b) \hspace{3cm}
(c)
\end{center}
\end{example}

We consider a modification  of face-firing called {\bf face-firing with a marked face}. Fix a marked face $f_0$ and an integer $n$, the weight of $f_0$.
The marked face acts simultaneously as a source and a sink, while eschewing some of the properties of both.
We use the usual face-firing rules away from $f_0$, but adapt the rules when firing into or out of $f_0$ in the following way: if $f$ is adjacent to $f_0$ and has weight $\K_f < n$, we can fire from $f_0$ to $f$. 
If on the other hand $\K_f > n$, we can fire from $f$ to $f_0$. 
If $\K_f = n$, then we cannot fire between $f_0$ and $f$. Notably, for any configuration $\K$, the value $\K_{f_0}$ never changes as a result of firing. Example \ref{ex: face-firing with a marked face} illustrates  face firing moves with a distinguished face.

\begin{example}\label{ex: face-firing with a marked face}
Below are three configurations on the grid complex with a marked face (shaded).  To obtain the second configuration (b) from the first configuration (a), we fire the marked face to the left.
This firing is only allowed because we are firing the marked face.
In the second picture (b), the values of the two faces involved are the same, so we cannot fire the marked face to the left.  The third picture (c) is obtained from the second configuration (b) by firing the top middle box to the right. 
\[
\begin{ytableau}
\textcolor{white}{.}    &  3 &    \\
   1 & *(teal!40) 2 &   \\
    &   &   \\
\end{ytableau}
\hspace{2cm}
\begin{ytableau}
\textcolor{white}{.}  &  3 &    \\
  2 & *(teal!40) 2 &  \\
    &   &     \\
\end{ytableau}
\hspace{2cm}
\begin{ytableau}
\textcolor{white}{.}    &  2 & 1   \\
   1 & *(teal!40) 2 &   \\
    &   &   \\
\end{ytableau}
\]
\begin{center}
(a) \hspace{3cm}
(b) \hspace{3cm}
(c)
\end{center}
 \end{example}

The {\bf (total) weight} of a configuration $\K$ is the sum of the non-zero face weights in $\K$:
\[\weight(\K):= \sum_{f} \K_{f}\]
and the {\bf support radius} of $\K$ is $ \len(\K): = \max \{ \dist(f, f_0): \K_f \neq 0\}.$

There are two configurations that will play a significant role in our analysis. The first configuration is the {\bf Aztec diamond}, which is the configuration whose faces $f$ have the following weights
\[
\az(n)_{f} := \text{max}\{n - \dist(f_0, f) + 1, 0\}
\quad
\text{for all faces }f \not= f_0\,,
\]
and $\az(n)_{f_0} = n$. 

The second configuration is the {\bf pulse of height $n$ and radius $r$}, which we denote by $\K(n,r)$. In the edge representation, $\K(n,r)$ can be described as a simple closed curve of radius $r$ around $f_0$, with $n$ units of flow oriented clockwise. In the face representation, $\K(n,r)$ is the configuration such that $\K(n,r)_f = n$ if $\dist(f_0,f)\leq r$ and $\K(n,r)_f = 0$, otherwise. Examples of $\K(4,2)$ and $\az(2)$ are shown below.  In the edge representation, $\K(4,2)$ is a closed curve with $4$ units of flow along the curve and $\az(2)$ has $1$ unit of flow on every edge in the discrete $\ell_1$-ball centered at the distinguished face.

\[
\scalebox{0.85}{
\begin{ytableau}
\textcolor{white}{.} & &  &  &  & & \\
&  &  & *(violet!40!white)4&  & & \\
&  &  *(violet!40!white) 4&  *(violet!40!white) 4 & *(violet!40!white) 4&  & \\
 & *(violet!40!white)4 &  *(violet!40!white) 4 & *(teal!40) 4 &  *(violet!40!white)4 & *(violet!40!white)4 & \\
&  &  *(violet!40!white) 4& *(violet!40!white) 4 &  *(violet!40!white) 4&&  \\
&  &  & *(violet!40!white)4  &  & & \\
 &  &  &  &  & & \\
\end{ytableau}
\hskip2cm 
\begin{ytableau}
\textcolor{white}{.} & &  &  &  & & \\
&  &  & *(violet!40!white)1&  & & \\
&  &  *(violet!40!white) 1&  *(violet!40!white) 2 & *(violet!40!white) 1&  & \\
 & *(violet!40!white)1 &  *(violet!40!white) 2 & *(teal!40) 2 &  *(violet!40!white)2 & *(violet!40!white)1 & \\
&  &  *(violet!40!white) 1& *(violet!40!white) 2 &  *(violet!40!white) 1&&  \\
&  &  & *(violet!40!white)1  &  & & \\
 &  &  &  &  & & \\
\end{ytableau}}
\]
\begin{center}
 \hspace{1cm} Pulse $\K(4,2)$ \hspace{2.2cm} Aztec diamond $\az(2)$
\end{center}

The significance of these two configurations originates with the following theorem, which proved the first instance of confluent behavior in the context of flow-firing.

\begin{theorem}[{{\cite[Theorem 9]{flowfiringprocesses}}}]
\label{thm:pulse}
Every sequence of flow-firing moves on $\K(n,0)$ terminates in $\az(n)$.
\end{theorem}

The three regimes theorem for flow-firing (Theorem \ref{thm:3regimes}) provides a fuller description of the possible behavior of initial configurations around a marked face, in particular showing that the pulse is essentially the only initial configuration that gives a confluent system.  

For context and comparison, we give the precise statements of the graphical theorem  before proceeding with the proof of the flow-firing theorem.

\begin{theorem}[{{\cite[Theorem 3.3]{bjorner-lovasz-shor}}}: Three Regimes for Graphical Chip-Firing]
Let $G$ be a finite connected graph with $n$ vertices and $m$ edges.
Let $\bf{c}$ be a configuration with $N$ chips in total.
Then
\begin{enumerate}
\item If $N < m$ then the chip-firing process terminates after finitely-many firing. 
\item If $m \leq N \leq 2m-n$ then there exists an initial configuration which
stabilizes and also one which does not.
\item If $N > 2m-n$ then the chip-firing process never stabilizes.
\end{enumerate}
\label{BLS}
\end{theorem}


\section{Path-Firing}\label{sec:pathFire}

A \textbf{path} $\PPP $ is a collection of successive faces  $\PPP  = (f_1,\dots,f_k)$ such that $ \dist(f_{i-1}, f_{i})=1$ for $2\leq i\leq k$. A flow-firing process involving only the faces of $\PPP$ is called {\bf path-firing}.
A path is called {\bf standard} if each $f_i$ has value $\max\{0, n-\dist(f_0,f_i)+1\}$. In other words, a standard path is one in which the weight of $f$ matches $\az(n)_f$.


Let $\R$ be a path configuration on $\PPP$ such that the face weights of $\R$ form a weakly decreasing sequence  $\R_1 \geq \dots \geq \R_k$. A configuration obtained by path-firing on $\R$ along the path $\PPP$ is called an {\bf intermediate} path configuration; if there are no more path-firing moves left, such a configuration is called {\bf path-stable}. Note that the face weights of any intermediate path configuration always form a weakly decreasing sequence. This means that firing from face $f_i$ can only be directed towards face $f_{i+1}$---in other words, firing can only happen in one direction along $\PPP$.

It turns out that path-firing is equivalent to graphical chip-firing along a line, which has been studied in \cite{BallsWalls}; see also \cite[Chapter 5]{klivans-book}. We will thus be able to apply standard results from graphical chip-firing to study path-firing. The key ingredient to forge this connection is the notion of \textbf{discrete derivatives}, defined below.  We thank an anonymous referee for pointing out this connection. 

Let $\R$ be a path configuration on a path $\PPP$ as given above. Write $d(\R)=(d_1,\dots,d_k)$ to denote the \textbf{discrete derivative} of $\R$, where $d_k=\R_k$ and $d_i = \R_i - \R_{i+1}$ for $i=1,\dots,k$.

Consider the configuration $\R'=(\R_1, \ldots, \R_i-1,\R_{i+1}+1,\ldots, \R_k)$ obtained from $\R$ by firing from face $f_i$ towards face $f_{i+1}$. Then the discrete derivative of $\R'$ is
\begin{equation}\label{eq:discrete.derivative}
  d(\R')= \begin{cases}
       (d_1, \ldots, d_{i-1}+1, d_i-2, d_{i+1}+1,\ldots, d_k)& \text{ if } i>1,\\
      (d_1-2, d_2+1,d_3,\ldots, d_k) & \text{ if } i=1.
  \end{cases}  
\end{equation}
 Notice that $\R$ can be recovered from its discrete derivative by taking partial sums: the weight $\R_i$ is given by 
 \[ \R_i = \sum_{j = i}^k d_j. \] This leads to the following, which follows immediately from Equation \eqref{eq:discrete.derivative}. 

 \begin{proposition}\label{prop:discretederivative}
     Path-firing on $\R$ along $\PPP$ is equivalent to chip-firing on a line with vertices $v_0,v_1,\ldots, v_k$, where  for any $0 \leq i \leq k$, vertex $v_i$ has $d_i$ chips, and vertex $v_0$ is a sink vertex (with no chips). 
 \end{proposition}

\begin{example} \label{example:configurationR.ell}
    Consider a path configuration $\R(\ell)$ for a fixed integer $\ell$, with $\R(\ell)_i = n$ for all $1 \leq i \leq \ell$. Then the discrete derivative $d(\R(\ell)) = (d_0, d_1, \cdots)$ of $\R(\ell)$ 
    \[ d(\R(\ell)) = (0, \cdots, 0,n, 0,\cdots 0)\]
    The corresponding graphical configuration has $d_{\ell} = n$ and $d_i = 0$ for all $i \neq \ell.$ This special case of graphical chip-firing is studied in \cite{BallsWalls} and its final configuration is well-understood (see Corollary \ref{lem:every-face-fires}).
\end{example}

Our ultimate goal is to understand configurations obtained from flow-firing on $\K(n,r)$. To do so, we will first analyze path-firing along a single row of $\K(n,r)$, which we implicitly introduced in Example \ref{example:configurationR.ell}. Explicitly, for any $\ell \in \Z_{\geq 0}$, define
\[ \PPP(\ell) := (f_1\ldots, f_{\ell+ \lfloor\frac{n}{2}\rfloor}),\]
and a configuration $\R (\ell)$ on $\PPP(\ell)$:
\[ \R(\ell):= 
\underbrace{(n,n, \dots, n}_{\ell~\textrm{times}},0,0, \dots)\,.
\]
Since $\PPP (\ell)$ does \emph{not} include  $f_0$, the number of chips stays constant while firing along $\PPP (\ell)$.
The discrete derivative of $\R(\ell)$ is 
\[
d(\ell) := 
\underbrace{(0,0, \dots, 0}_{\ell-1~\textrm{times}},n,0,0, \dots )\,.
\]

By Proposition \ref{prop:discretederivative}, flow-firing on a path with initial configuration $\R(\ell)$ is the same as graphical chip-firing on a line with $\ell + \lfloor\frac{n}{2}\rfloor + 1$ vertices (equivalently $\ell +\lfloor\frac{n}{2}\rfloor $ edges) where the first vertex is a sink. 

Chip-firing in the above context will terminate \cite[Theorem 2.5.2]{klivans-book}, but we can be more precise about its termination behavior using \cite[Theorem 1]{BallsWalls}. Below, we translate this behavior back into the language of flow-firing, so that we can apply it later in Section \ref{sec:3-regimes-2-aztec}.

\begin{cor}[{{\cite[Theorem 1]{BallsWalls}}}]\label{lem:every-face-fires}
For any $\ell \geq 0$, the path that leads to $\R(\ell)$ along $\PPP (\ell)$ terminates uniquely. Let $\R^{\circ}$ be this terminal configuration, and let $S_\ell$ be the set of faces that fire in the process of obtaining $\R^{\circ}$ from $\R(\ell)$. Then   
    \begin{itemize}
\item If $\ell > \lfloor \frac{n}{2} \rfloor$ then $S_\ell = \{ f_{k+1}, f_{k+2}, \dots, f_{r-1}\}$,
\item If $\ell \leq \lfloor \frac{n}{2} \rfloor$, then $S_\ell = \{ f_1, \dots,f_{r-1} \}$,
\end{itemize}
where $k = \ell - \lfloor \frac{n}{2} \rfloor$ and let $r = \len(\R^\circ)$. Moreover, if $\ell \leq \lceil \frac{n}{2} \rceil$ 
    then $\len (\R^\circ) = \ell + \lfloor \frac{n}{2} \rfloor$. 
\end{cor}

As an example, we compute the discrete derivative $d(\R^\circ)$ corresponding to the configuration $\R^\circ$ in Corollary \ref{lem:every-face-fires},  when $\ell \leq \lceil \frac{n}{2}\rceil$:
\begin{equation}
    d(\R^\circ) = \begin{cases}
        \underbrace{(1,1, \dots, 1}_{\ell + \lfloor\frac{n}{2}\rfloor~\textrm{times}}, 0, \dots ) &\textrm{ if } n \textrm{ is odd, }\\
        &\\
        \underbrace{(1,1, \dots, 1}_{\ell-1~\textrm{times}},0,\underbrace{1,1, \dots,1 }_{\frac{n}{2} ~\textrm{times}}, 0,\dots) & \textrm{ if } n \textrm{ is even}.
    \end{cases}
\end{equation}
As discussed above, $d(\R^\circ)$ completely determines $\R^\circ$. Determining $d(\R^\circ)$ when $\ell> \lceil \frac{n}{2}\rceil$ is even simpler and is given explicitly in \cite{BallsWalls}.

In addition to describing the termination behavior of $\R(\ell)$, Proposition \ref{prop:discretederivative} allows us to better understand intermediate path configurations more generally. Lemma \ref{lem:adjacent-triple} below implies that the entries of discrete derivative cannot be zero ``too often'' unless it started off that way.
We omit the proof because it is similar to the proof of \cite[Lemma 1]{BallsWalls}.

\begin{lemma}\label{lem:adjacent-triple}
Let $\R$ be a path configuration on a path $\PPP$ whose face weights form a weakly decreasing sequence. 
Let $\R'$ be an intermediate path configuration obtained from path-firing on $\R$. Let $d(\R') = (d_0, \cdots, d_k, \cdots)$ denote the discrete derivative of $\R'$.  Then 
\begin{enumerate}
    \item[(a)] if $d_i=d_{i+1}=0$, the faces $f_i, f_{i+1}$ and $f_{i+2}$ never fire, and
    \item[(b)]if $d_i = d_{j} = 0$  for $i+1<j$ and  $d_\ell = 1$ for $i < \ell < j$ then the faces $f_{i}, f_{i+1}, f_{j},$ and $f_{j+1}$ never fired. 
\end{enumerate}
\end{lemma}

Lemma \ref{lem:adjacent-triple} significantly limits the types of configurations that can arise from path-firing on along $\PPP$, and will be useful later in Section \ref{regime_3:nonuniqueterminate}.


\section{The first regime}\label{sec:regimes}

\begin{reptheorem}{thm:3regimes}[Part \ref{3regime1}]
If $r \leq 1$, then firing from $\K(n,r)$ terminates uniquely in the Aztec diamond $\az(n)$.
\end{reptheorem}

Our approach to showing a configuration \emph{can} terminate in $\az(n)$ (e.g. Regimes 1 and 2) relies on finding an intermediate configuration of $\K(n,r)$ that does not ``violate" the Aztec diamond. A configuration $\K$ {\bf violates} the Aztec diamond $\az(n)$ if some face $f$ has $ \K_{f} > \az(n)_{f}$. After finding such a configuration, Lemma \ref{lem:sub-configurations} allows us to conclude that the initial configuration $\K(n,r)$ \emph{can} terminate in $\az(n)$.

\begin{lemma}\label{lem:sub-configurations}
Let $\K$ be a configuration with $\K_{f_0}=n$.
If $\K_f \leq \az(n)_f$ for all faces $f$,
then there is a sequence of firing moves sending $\K$ to $\az(n)$.
\end{lemma}
\begin{proof}
The core of this proof comes from two observations.
\begin{enumerate}
    \item Let $f,g$ be adjacent faces with $g$ strictly further from $f_0$ than $g$, so that $\dist(f_0,f) + 1 = \dist(f_0,g)$.
    If $K_f = \az(n)_f$ and $K_g < \az(n)_g$, then $K_g \leq K_f - 2$ and hence we can fire from $f$ to $g$.
    \item If $g$ is adjacent to $f_0$ and $K_g < \az(n)_g$, then $K_g \leq n - 1$ and we can fire from $f_0$ to $g$.
\end{enumerate}
Let $g$ be a face such that $K_g < \az(n)_g$ but all faces strictly closer to $f_0$ satisfy $K_f = \az(n)_f$. 

If $g$ is adjacent to $f_0$, then by (2) above, we can fire from $f_0$ to $g$. If $g$ is not adjacent to $f_0$, then there is some face $f$ adjacent to $g$ and strictly closer to $f_0$; by (1) we can fire from $f$ to $g$.
Repeat this process until $\K_g = \az(n)_g$ for all faces $g$.
\end{proof}

\begin{proposition} \label{prop:3-regimes-1}
If $r \leq 1$, then $\K(n,r)$ terminates uniquely in the Aztec diamond.
\end{proposition}

\begin{proof}
The $r = 0$ case was proven in \cite[Theorem 9]{flowfiringprocesses}. If $r = 1$, then $\K(n,1)$ is an intermediate configuration obtained from  $\K(n,0)$ and it terminates uniquely at $\az(n)$ by \cite[Theorem 9]{flowfiringprocesses}.
\end{proof}

We will see in the subsequent section that although there are many more configurations that \emph{can} terminate in the Aztec diamond, in general this termination will not be unique. 


\section{The Second Regime}\label{sec:regime2}

\begin{reptheorem}{thm:3regimes}[Part \ref{3regime2-a}]
If $2 \leq r\leq \lceil\frac{n}{2}\rceil$, flow-firing on $\K(n,r)$ does not terminate uniquely. 
\end{reptheorem}

In order to prove this, we will show that there exists (i) a sequence of firings that terminate in $\az(n)$ and (ii) a sequence of firings that terminate in a configuration that is not $\az(n)$.

In what follows, we state our results for Regime 2 in terms of $\K(n,r)$. However,
both claims in Theorem \ref{thm:3regimes}(\ref{3regime2})
hold more generally for a closed curve with weight $n$ strictly containing $\K(n,1)$ and contained within $\K(n,r)$ for $2 \leq r \leq \lceil\frac{n}{2}\rceil$.


\subsection{Second regime: termination violating the Aztec diamond.}\label{sec:3-regimes-2-violate}
Our goal in this section is to show that the configuration $\K(n,r)$ need not terminate in the Aztec diamond when $2 \leq r \leq \lceil \frac{n}{2}\rceil$.

\begin{lemma}\label{lemma:free-to-fire}
Let $\K$ be any configuration and let $f$ and $g$ be adjacent faces with
\[
\dist(f_0, f) < \dist(f_0, g) < n\,.
\]
If $f$ violates the Aztec diamond but $g$ does not, then one can fire from $f$ to $g$. In the resulting configuration $\K'$ , $\K'_f$ will violate the Aztec diamond if $\K_f > \az(n)_f + 1$, and
$\K'_g$ will violate the Aztec diamond if $\K_{g} = \az(n)_{g}$.
\end{lemma}
\begin{proof}
To see the first part of the lemma, we need to show that $\K_f - \K_g \geq 2$.
Recall that 
\[ \az(n) = \max \{ n - \dist(f_0, f) + 1, 0 \}. \] 
Since $\dist(f_0, f) < \dist(f_0, g) < n$, we have that $\az(n)_{g} < \az(n)_{f}$. Since they are adjacent,
\[ \az(n)_f - \az(n)_g = 1.\]
Thus 
\[
\K_f \geq \az(n)_{f} + 1 = \az(n)_{g} + 2.
\]
On the other hand, since $\K_g$ does not violate the Aztec diamond, $\K_g \leq \az(n)_g$. 
 Thus 

\[
\K_f \geq \K_g + 2
\]
as desired. The second part of the claim is clear from inspection.
\end{proof}

\begin{lemma}[Flooding lemma]\label{lem:flooding-lemma}
Let $f$ be a face such that $\K_f>\az(n)_f > 0$.
Then there is configuration $\K'$, reachable from $\K$ by flow-firing moves, and a face $g$ such that
\begin{itemize}
    \item $\dist(f_0, g) = 1 + \dist(f_0, f)$ and 
    \item $\K'_{g} > \az(n)_{g}$.
\end{itemize}
\end{lemma}
\begin{proof}
We explicitly construct $\K'$ as in the statement of the lemma. Let $V$ be the set of faces that violate the Aztec diamond in the configuration $\K$.
Take $f \in V$ to be a face with maximal distance from $f_0$. By the maximality of $f$, every neighbor $g$ of $f$ with $\dist(f_0, g) = 1 + \dist(f_0, f)$ does \emph{not} violate the Aztec diamond. Choose such a neighbor $g$ and a path $\PPP = (f_0, f_1, \cdots, f_{\ell} = g)$ that does not pass through $f$.

We claim that one can fire along $\PPP$ until $g$ violates the Aztec diamond. There are two cases:
\begin{itemize}
    \item \emph{Case I: For $1 \leq i \leq \ell$, no $f_i$ in $\PPP$ violates the Aztec diamond.}\\
    First fire along $\PPP$ to create a standard path. This is always possible because $f_i$ never violates the Aztec diamond. Note that after this process, $g$ has value $\az(n)_{g}$. By Lemma~\ref{lemma:free-to-fire}, we can fire from $f$ into $g$, which means that $g$ will then violate the Aztec diamond.

    \item \emph{Case II: There is some $1\leq i\leq \ell$ where $f_i$ violates the Aztec diamond.}\\
    Pick the violating face $f_i$ that is furthest from $f_0$. Thus $f_{i+1}$ does not violate $\az(n)$, and so by Lemma \ref{lemma:free-to-fire} we can fire from $f_i$ to $f_{i+1}$. Repeat this process; it necessarily will terminate because we are in the conservative flow setting. After termination, either we are in \emph{Case I} or the violation has been pushed to $\PPP_{\ell} = g$. 
\end{itemize}
Thus, the resulting configuration $\K'$ has $\K'_{g} > \az(n)_{g}$ since $g$ violates the Aztec diamond.
\end{proof}

The Flooding Lemma implies that for a configuration $\K(n,r)$ with $r$ large enough, one can construct a terminal configuration that violates $\az(n)$. 

\begin{reptheorem}{thm:3regimes}[Part \ref{3regime2-b}]
Given a configuration $\K(n,r)$ with $2 \leq r\leq \lceil\frac{n}{2}\rceil$, there is a stable configuration $\K^*$ such that $\K^* \neq \az(n)$.
\end{reptheorem}

\begin{proof}
Since $2 \leq r\leq \lceil\frac{n}{2}\rceil$, there exists at least one face violating $\az(n)$.
Let $V$ be the set of all such faces, and choose $f \in V$ of maximal distance from $f_0$.
Since $f$ is maximal, every face $g$ with $\dist(g,f_0) > \dist(f,f_0)$ has $\K_g \leq \az(n)_g$.
Fix such a $g$ with $\dist(g,f_0) = \dist(f,f_0) + 1$.
Applying the Flooding Lemma (Lemma \ref{lem:flooding-lemma}) to $f$ and $g$ gives a new configuration $\K'$ such that some face $g$ with $\dist(g,f_0) > \dist(f,f_0)$ has $\K_g > \az(n)_g$.

Repeat this process until the violating face $h$ has $\dist(f_0,h) > n$, so that  $\az(n)_h = 0$; note that this is always possible by Lemma \ref{lemma:free-to-fire}. Call this configuration $\K'$. Since $\K'_h > 0$, no firing moves on $\K'$ can ever decrease the weight of face of $h$, and so firing on $\K'$ will result in a terminal configuration $\K^*$ with $\K^{*}_h > 0$. Hence $\K^* \neq \az(n)$. 
\end{proof}

\begin{example}\label{ex: secondregimeproof}
Below we draw four configurations.
The first (a) is $\K(3,2)$, and the other three are obtained from the first by successive firing moves.
Using the notation of Lemma \ref{lem:flooding-lemma}, for each of the configurations $\K$ below, pick $f$ to be the northern-most nonzero face that can fire in $\K$ (shaded, in pink), and let $g$ be the neighboring face (shaded, in yellow).
Using these choices of $f$ and $g$, we show a sequence of three iterations of the Flooding Lemma.
\[
\scalebox{0.9}{
\begin{ytableau}
\textcolor{white}{.} &   &  &  &  \\
 & *(yellow!40!white)  & *(magenta!20!white)3 &  &  \\
 &   3 &  3 & 3 &   \\
 3 &  3 & *(teal!40) 3 &  3 &  3 \\
  &  3 &  3 &  3 &   \\
  & &  3 & &  \\
\end{ytableau}
\hskip0.5cm
\begin{ytableau}
\textcolor{white}{.} &   &  &  &  \\
 & *(yellow!40!white) 1 &  2 &  &  \\
 & *(magenta!20!white) 3 &  3 & 3 &   \\
 3 &  3 & *(teal!40) 3 &  3 &   \\
  &  3 &  3 &  3 &   \\
  & &  3 & &  \\
\end{ytableau}
\hskip0.5cm
\begin{ytableau}
\textcolor{white}{.} &  *(yellow!40!white)  &  &  &  \\
 & *(magenta!20!white) 2 & 2 &  &  \\
 &  2 &  3 & 3 &  \\
 3 &  3 & *(teal!40) 3 &  3 &  3 \\
  &  3 &  3 &  3 &   \\
  & &  3 & &  \\
\end{ytableau}
\hskip0.5cm
\begin{ytableau}
\textcolor{white}{.} &  *(magenta!20!white)  1 &  &  &  \\
& 1 & 2 &  &  \\
 &  2 &  3 & 3 &   \\
 3 &  3 & *(teal!40) 3 &  3 &  3 \\
  &  3 &  3 &  3 &   \\
  & &  3 & &  \\
\end{ytableau}}
\]
\begin{center}
(a) \hspace{2.3cm} (b) \hspace{2.3cm} (c) \hspace{2.3cm} (d)
\end{center}
\end{example}

\subsection{Second regime: termination in the Aztec diamond}\label{sec:3-regimes-2-aztec}
Having constructed a stable configuration from $\K(n,r)$ which is not equal to $\az(n)$, we now wish to show that $\az(n)$ \emph{can} appear as a final configuration for $\K(n,r)$ when $ 2 \leq r \leq \lceil \frac{n}{2} \rceil $. Our goal is to show:

\begin{reptheorem}{thm:3regimes}[Part \ref{3regime2-a}]
Given a configuration $\K(n,r)$ with $2 \leq r\leq \lceil\frac{n}{2}\rceil$, there is a stable configuration $\K^*$ obtained from firing on $\K(n,r)$ such that $\K^* = \az(n)$.
\end{reptheorem}

\begin{proof}
Divide $\K(n,r)$ and $\az(n)$ into four quadrants, as shown below in the case of $\K(4,2)$. Each quadrant is given in a different color, and $\az(4)$ is shaded in gray in the background. Note that the marked face $f_0$ is not included in any of the quadrants.
\[
\scalebox{0.8}{
\begin{ytableau}
\textcolor{white}{.} &  & &  &  *(black!20!white)  &    &  & & \\
 & &    &  *(black!20!white) & *(black!20!white) & *(black!20!white)    &  &   & \\
 &   &*(black!20!white) & *(black!20!white)  & *(violet!40!white) 4 & *(black!20!white)  & *(black!20!white)  &   & \\
  &  *(black!20!white)  &*(black!20!white) &  *(yellow!20!white) 4 &  *(violet!40!white) 4 & *(violet!40!white) 4 &  *(black!20!white)  &   *(black!20!white) & \\
 *(black!20!white)  &  *(black!20!white)  &*(yellow!20!white) 4 & *(yellow!20!white) 4 & *(teal!85!white) 4 & *(red!30!white) 4 &  *(red!30!white) 4  &   *(black!20!white) & *(black!20!white) \\
  &   *(black!20!white) & *(black!20!white) &  *(green!20!white)4 &  *(green!20!white)4 &  *(red!30!white)4 &   *(black!20!white)  &  *(black!20!white)  &\\
  &   & *(black!20!white) &*(black!20!white)  &  *(green!20!white) 4 & *(black!20!white) &  *(black!20!white)  &   &\\
& &    & *(black!20!white) & *(black!20!white)&*(black!20!white)    &  &  &  \\
& &    &  &*(black!20!white) &    &  & &  \\
\end{ytableau}}
\]

Without loss of generality, we can study a single quadrant and then appeal to symmetry to describe the remaining quadrants.  Consider the rows in the first quadrant of $\K(n,r)$ (colored violet in the figure above) and count them from top to bottom. Let $\PPP (\ell)$ be the path that is obtained from extending the $\ell^{\text{th}}$ row of the quadrant to the boundary of $\az(n)$. The relevant configuration on $\PPP(\ell)$ is then 
\[ \R(\ell):= 
\underbrace{(n,n, \cdots, n}_{\ell~\textrm{times}},0, \dots)\,.
\]
whose behavior was studied in Section \ref{sec:pathFire}. 

Note that $\dist(f_0, f_i)= r-\ell+i$ for each face $f_i$ of $\PPP (\ell)$ and $$\az(n)_{f_i} = n- \dist(f_0,f_i)+1 = n-r+\ell-i+1.$$ 
Let $\R^{\circ}$ be the unique path-stable configuration obtained from path-firing on $\R (\ell)$ given in Corollary \ref{lem:every-face-fires}. Translating the discrete derivative in Corollary \ref{lem:every-face-fires} into the language of flow-firing, we have that $\R^{\circ}_i$, the value of face $f_i$ on $\R^{\circ}$, is 
\[ \R^{\circ}_i = \begin{cases} 
\ell + \lfloor n/2 \rfloor - i & \text{if }n \textrm{ even and } i < \ell, \\
\ell + \lfloor n/2 \rfloor - i + 1 &  \textrm{ else}.
\end{cases}\]
Since $r \leq \lceil n/2 \rceil$,
it follows that 
\[ \az(n)_{f_i} \geq \lfloor n/2 \rfloor + \ell - i + 1 \geq  \R^{\circ}_i. \]

Thus, the value of $\R^{\circ}_i$ cannot exceed $\az(n)_{f_i}$ along $\PPP (\ell)$. Firing $\K(n,r)$ along all rows in the analogous way thus yields a configuration which does not violate $\az(n)$. Thus by Lemma \ref{lem:sub-configurations}, there is a sequence of firing moves that can then be implemented to obtain $\az(n)$ as a stable configuration. 
\end{proof}

\begin{example}\label{ex:regime2az}

The full configurations after path-firing in all four quandrants is shown below. 
Note that the second configuration does not violate the Aztec diamond $\az(4)$.
This means that we can apply Lemma \ref{lem:sub-configurations} to obtain $\az(4)$ from the second by firing moves.
\[
\scalebox{0.8}{
\begin{ytableau}
\textcolor{white}{.} &  & &  &  *(black!20!white)  &    &  & & \\
 & &    &  *(black!20!white) & *(black!20!white) & *(black!20!white)    &  &   & \\
 &   &*(black!20!white) & *(black!20!white)  & *(violet!40!white) 4 & *(black!20!white)  & *(black!20!white)  &   & \\
  &  *(black!20!white)  &*(black!20!white) &  *(yellow!20!white) 4 &  *(violet!40!white) 4 & *(violet!40!white) 4 &  *(black!20!white)  &   *(black!20!white) & \\
 *(black!20!white)  &  *(black!20!white)  &*(yellow!20!white) 4 & *(yellow!20!white) 4 & *(teal!85!white) 4 & *(red!30!white) 4 &  *(red!30!white) 4  &   *(black!20!white) & *(black!20!white) \\
  &   *(black!20!white) & *(black!20!white) &  *(green!20!white)4 &  *(green!20!white)4 &  *(red!30!white)4 &   *(black!20!white)  &  *(black!20!white)  &\\
  &   & *(black!20!white) &*(black!20!white)  &  *(green!20!white) 4 & *(black!20!white) &  *(black!20!white)  &   &\\
& &    & *(black!20!white) & *(black!20!white)&*(black!20!white)    &  &  &  \\
& &    &  &*(black!20!white) &    &  & &  \\
\end{ytableau}}
\hskip1cm
\scalebox{0.8}{
\begin{ytableau}
\textcolor{white}{.} &  & &  &  *(black!20!white)  &    &  & & \\
 & &    &  *(black!20!white) & *(black!20!white) & *(black!20!white)    &  &   & \\
 &  &   *(black!20!white) &  *(black!20!white)& *(teal!40!white)2 & *(teal!40!white)1 &  *(teal!40!white)1&  & \\
&  *(black!20!white)1 &   *(black!20!white)1 & *(black!20!white)2 & *(violet!40!white)3 &  *(violet!40!white)2 &  *(violet!40!white)2& *(violet!40!white)1 &  \\
*(black!20!white)1 & *(black!20!white) 2 &   *(black!20!white)2 &  *(black!20!white)3& *(teal!85!white)  4&  *(black!20!white)3  &*(black!20!white)2  & *(black!20!white) 2  & *(black!20!white) 1\\
& *(black!20!white)1&   *(black!20!white)2 & *(black!20!white)2 & *(black!20!white)3& *(black!20!white)2   &*(black!20!white)1  & *(black!20!white) 1 & \\
& &  *(black!20!white) 1 &*(black!20!white) 1 & *(black!20!white)2& *(black!20!white)  & *(black!20!white) &  &  \\
& &    & *(black!20!white) & *(black!20!white)&*(black!20!white)    &  &  &  \\
& &    &  &*(black!20!white) &    &  & &  \\
\end{ytableau}}
\]
\end{example}


\section{The Third Regime}\label{sec:regime_3}
Finally, we turn to the third regime. We first show that when $r > \lceil n/2 \rceil$, $\K(n,r)$ cannot terminate uniquely in Section \ref{regime_3:nonuniqueterminate}. Then, we show that for $r > \lceil \frac{n}{\sqrt{3}} \rceil +1$, it is impossible for $\K(n,r)$ to terminate in the Aztec diamond in Section \ref{sec:thirdregimeaztermination}.

\subsection{Non-unique termination}\label{regime_3:nonuniqueterminate}

We will show that $\K(n,r)$ does not terminate uniquely for $r > \lceil n/2 \rceil$. Our proof relies on decomposing the grid in two distinct ways, shown below:

        \begin{figure}[H]
    \centering
  \scalebox{0.8}{    \includegraphics[width=0.5\textwidth]{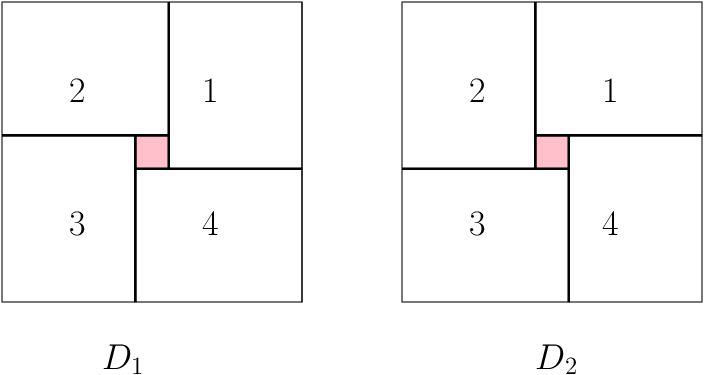}}
    \caption{Two distinct ways to decompose the grid into quadrants.}
    \label{fig:quad}
\end{figure}
\vspace{-0.5cm}
For $i=1,2$, we refer to $D_i$ as the decomposition of the grid. Our argument introduces an algorithm that fires identically within each quadrant, but produces two different final cofigurations using the two decompositions in Figure \ref{fig:quad}.

\begin{reptheorem}{thm:3regimes}[Part \ref{3regime3-a}]
If $r > \lceil n/2 \rceil$ then $\K(n,r)$ does not terminate uniquely.
\end{reptheorem}
\begin{proof}It is immediate that $\K(2,r)$ does not terminate uniquely when $r\geq 1$. So, suppose $n>2$. For the sake of contradiction, assume $\K(n,r)$ terminates uniquely. 

Consider the following flow-firing algorithm on $\K(n,r)$:
\begin{enumerate}
    \item Divide $\K(n,r)$ into four quadrants as in either $D_1$ or $D_2$ as in Figure \ref{fig:quad};
    \item While possible, path-fire within each quadrant along all horizontal paths (e.g. rows) away from $f_0$;
    \item While possible, path-fire within each quadrant along all vertical paths (e.g. columns) away from $f_0$;
    \item Return to step (2).
\end{enumerate}

We claim that this process results in a stable configuration, regardless of whether one divides quadrants by $D_1$ or $D_2$. By construction, within each quadrant the resulting configuration is stable. Note that any adjacent faces between different quadrants must differ by at most one. This follows from the fact that the quadrants are reflections of one another, with face values shifted exactly one unit up (down) and right (left). Let $\K_i$ be the stable configuration obtained from the above process, with the choice of quadrants given by $D_i$ for $i = 1,2$.  

By assumption, $\K(n,r)$ terminates uniquely and so both the supports and the face weights of $\K_1$ and $\K_2$ are the same.  

We label the face weights of $\K_1$ and $\K_2$ as follows: first, label the faces of the grid by $(j,k) \in \Z^2$ so that $f_{j,k} \in \Z$ is the face weight at face $(j,k)$. Fix the indexing so that the distinguished face is $f_{0,0}$, and $(j,k)$ increase to the right and up. We will focus on the first quadrant of $\K_i$, denoted by $Q_i$.  Label the face weights in both $Q_i$ as $a_{j,k}^{(i)} \in \Z^{2}_{> 0}$, where the bottom-left face is $a_{1,1}^{(i)}$ for $i = 1,2$ and again indices increase right and up.

\begin{figure}[H]
\centering
\scalebox{0.95}{\begin{subfigure}{.33\textwidth}
  \centering
    \includegraphics[width=0.96\linewidth]{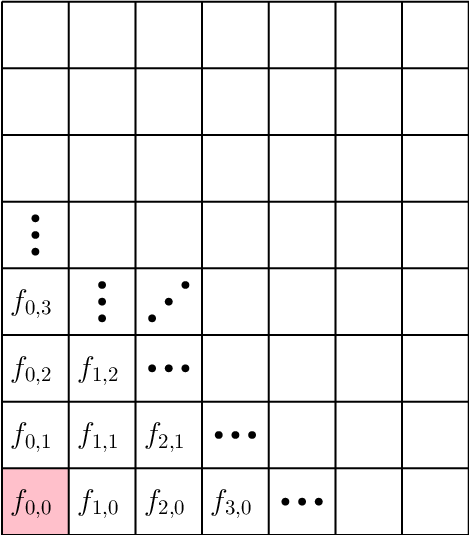}
  \label{fig:D1_face_weights}
\end{subfigure}}
\scalebox{0.95}{\begin{subfigure}{.33\textwidth}
  \centering
 \includegraphics[width=0.96\linewidth]{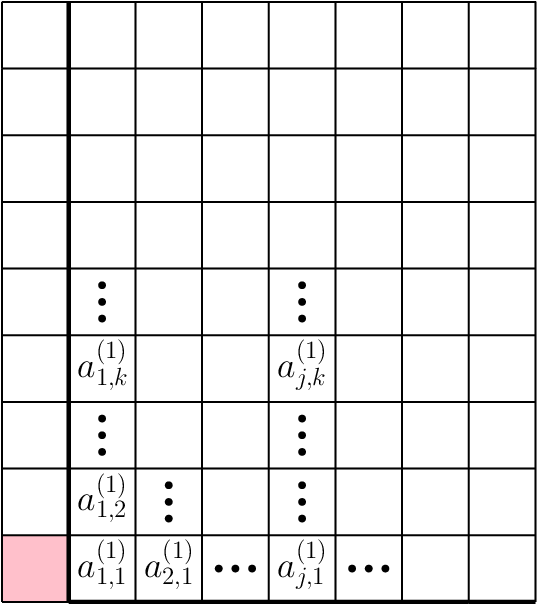}
\end{subfigure}}
\scalebox{0.95}{\begin{subfigure}{.33\textwidth}
  \centering
  \includegraphics[width=.84\linewidth]{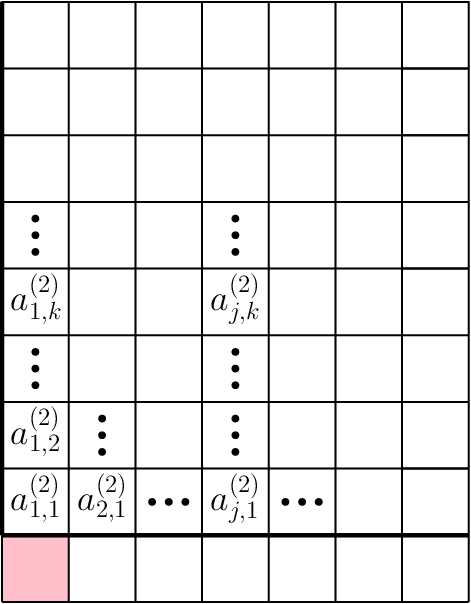}
  \label{fig:D2_weights}
\end{subfigure}}
\caption{The face labels of the grid, face weights in $Q_1$ and $Q_2$ (respectively)}
\label{fig:Q1 and Q2}
\end{figure}

 We make three observations:
    \begin{itemize}
        \item because the configurations $Q_1$ and $Q_2$ were obtained by the same firing process, $a_{j,k}^{(1)} =a_{j,k}^{(2)}$ for all $j,k \in \Z^2_{>0}$. In other words, $Q_2$ is obtained by shifting $Q_1$ up and left by 1; 
        \item by the definitions of $D_1$ and $D_2$, we have $a_{1,1}^{(1)}$ is the face weight of $f_{1,0}^{(1)}$ in $\K_1$, while  $a_{1,1}^{(2)}$ is the face weight of $f_{0,1}^{(2)}$ in $\K_2$; 
        \item  by the assumption that $\K_1 = \K_2$, we have $a_{j,k}^{(1)} = a_{j+1,k-1}^{(2)}$ for all $j,k$. Therefore  $a_{j,k}^{(1)} = a_{j-1,k+1}^{(1)}$ for all $j,k$.
    \end{itemize}
Label the columns of $Q_i$ by $C_1, C_2, \ldots, C_t$ (from left to right), and let the length of $C_j$ be $n_j$. Then by the last observation, we must have that $n_{j+1} = n_{j} - 1$ for $j=1,\ldots, t-1$. Thus, without loss of generality, we may describe $Q_1$ as in Figure \ref{fig:case3a} below:

    \begin{figure}[H]
    \centering
    \scalebox{0.95}{  \includegraphics[width=0.35\textwidth]{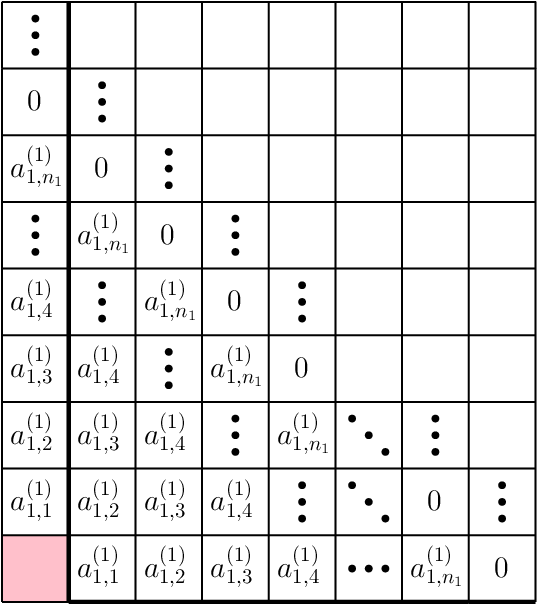}}
 \caption{The face weights in $Q_1$}
    \label{fig:case3a}
\end{figure}

From now on, we focus on $Q_1$. For ease of notation, let 
$a_{1,j} :=  a_{1,j}^{(1)} $ for $1\leq j \leq n_1$.  Note that  $a_{1,n_1}=1$ and $n_1>r$.

Write $r= \lceil n/2 \rceil+k$ for some $k\geq 1$. Since $n>2$, we have $r\geq k+2$.   Consider the faces $f_{1,0}, \cdots, f_{1,n_1-1}$ in the first column of $Q_1$ which have face value $a_{1,1}, \cdots, a_{1,n_1}$.

It follows from Corollary \ref{lem:every-face-fires} and the firing rules to obtain $Q_1$ that faces $f_{1,k+1}, \ldots, f_{1,n_1-2}$ all fired but faces $f_{1,0},\ldots, f_{1,k-1}$ may or may not have fired  in this process to obtain $Q_1$. Thus, $a_{1,i} \leq n-1$ for all $k+1\leq i\leq r+1$. In addition, if $a_{1,\ell}=n$ (in other words, $f_{1,\ell}$ never fired), then $k\geq \ell$.

The last observation we make is that the total face weight in $Q_1$ is
\begin{equation}\label{eq:facecount} \sum_{i,j} a_{i,j} = \frac{n r (r+1)}{2} \end{equation}
since face weight is preserved in the firing process to obtain $Q_1$.

Supposing the above conditions are satisfied, 
since $a_{1,r+1}\leq n-1$, there are three cases to consider. We will show that each case yields a contradiction, so that $\K(n,r)$ cannot terminate uniquely.

\textbf{Case 1:}  Suppose $a_{1,r+1}=n-1$. Then either $a_{1,r}=n$ or $a_{1,r-1}=n$ by Lemma \ref{lem:adjacent-triple}. This means either $k\geq r$ or $k \geq r-1$. Since $r\geq k+2$, neither situation is possible. Thus it is impossible that $\K_1 = \K_2$, and so $\K(n,r)$ cannot terminate uniquely in this case.

\textbf{Case 2:} Suppose $a_{1,r+1}=n-2$. Then, as in the previous case, either $a_{1,r-1}=n$ or $a_{1,r-2}=n$ by Lemma \ref{lem:adjacent-triple}. As we have seen in the previous case, $a_{1,r-1}=n$ is not possible.  Note that when $a_{1,r-2}=n$, we have either 
\begin{enumerate}
    \item[(i)] $(a_{1,r+1}, a_{1,r}, a_{1,r-1}, a_{1,r-2,})= (n-2, n-2, n-1, n)$ or
    \item[(ii)] $(a_{1,r+1}, a_{1,r}, a_{1,r-1}, a_{1,r-2})= (n-2, n-1, n-1, n)$.
\end{enumerate}
If $a_{1,r-2}=n$, then $k\geq r-2$ which implies that $r=k+2$. This happens only when $n=3$ or $n=4$.  By counting the total number of chips in $Q_1$ for $n=3$ and total face weight count given in \eqref{eq:facecount}, we get either
\begin{enumerate}
    \item[(i)] $\frac{3}{2} k(k+1)+4k+7= \frac{3}{2} (k+2)(k+3)$ or
    \item[(ii)] $\frac{3}{2} k(k+1)+5k+9= \frac{3}{2} (k+2)(k+3)$
\end{enumerate}
neither of which have non-negative integer solutions. Similarly, for $n=4$, we get
\begin{enumerate}
    \item[(i)] $2 k(k+1)+8k+17=2 (k+2)(k+3)$ or
    \item[(ii)] $2 k(k+1)+9k+97=2 (k+2)(k+3)$
\end{enumerate}
neither of which have non-negative integer solutions. Hence, this case is also not possible.

\textbf{Case 3:} Suppose $a_{1,r+1}\leq n-3$. 

We will now construct a third terminal configuration $\K_3$ from $\K(n,r)$, and show that it is impossible for $\K_1 = \K_2 = \K_3$.

To construct $\K_3$, consider the intermediate configuration which is obtained from $\K(n,r)$ by applying the algorithm described at the beginning of the proof, but only in the first quarter of $D_1$. This yields the configuration $Q_1$ in the first quadrant and $\K(n,r)$ in the other three quadrants.

Recall that $r > \lceil n/2 \rceil$, and consider the $(r+1)^{\text{th}}$ row of $Q_1$ (counted from south to north). Denote the configuration  on this row  as $\R$ where
\[ \R:= (a_{1,r+1}, a_{1,r+2}, \cdots, a_{1,n_1}, 0, \cdots).\]
Note that $f_{n_1+1,r}$ is the first face with weight 0 in $\R$.

Let $\tilde{\R}$ be the extension of $\R$ outside of $Q_1$ one unit to the left, i.e. by including the face $f_{0,r}$; this new face has weight $n$, since no faces in the fourth (i.e. north-west) quadrant have been fired yet:
\[ \tilde{\R}:= (n, a_{1,r+1}, a_{1,r+2}, \cdots, a_{1,n_1}, 0, \cdots).\]

Now path-fire on $\tilde{\R}$ along this row. Since $a_{1,r+1} \leq n-3$, we see that by applying Lemma \ref{lem:adjacent-triple}, after path-firing on $\tilde{\R}$ the face weight of $f_{n_1+1,r}$ is non-zero. Call this intermediate configuration $\tilde{\K}$.

Finally, we can obtain a stable configuration $\K_3$ from $\tilde{\K}$ by \cite[Theorem 4]{flowfiringprocesses}. Importantly, in $\K_3$ the face weight of $f_{n_1+1,r}$ is non-zero, because it is non-zero in $\tilde{\K}$. On the other hand, the weight of $f_{n_1+1,r}$ is zero in $\K_1$. Hence $\K_1 \neq \K_3$, and so $\K(n,r)$ cannot terminate uniquely.
\end{proof}

\begin{example}\label{ex:non-unique}
     In the case of $\K(3,3)$, the figures below show two distinct final configurations. We decompose the grid as in Figure \ref{fig:quad}, and perform firings as in the proof of Theorem \ref{thm:3regimes}(\ref{3regime3-a}).

\[
\scalebox{0.8}{
\begin{ytableau}
\textcolor{white}{.} &  & &  &  *(violet!40!white) 1 &    &  & & \\
 & &    &  *(yellow!20!white) 1& *(violet!40!white)2& *(violet!40!white)  1 &  &   & \\
 &  &   *(yellow!20!white)1 &  *(yellow!20!white)2& *(violet!40!white)2 & *(violet!40!white)2 &  *(violet!40!white)1&  & \\
&  *(yellow!20!white) 1&   *(yellow!20!white)2 & *(yellow!20!white)2 & *(violet!40!white)3 &  *(violet!40!white)3 &  *(violet!40!white)2& *(violet!40!white)1 &  \\
*(yellow!20!white) 1& *(yellow!20!white) 2&   *(yellow!20!white)3 &  *(yellow!20!white)3& *(teal!85!white)  3&  *(red!30!white)3  & *(red!30!white)3  & *(red!30!white)2 & *(red!30!white)1\\
& *(green!20!white)1&   *(green!20!white)2 & *(green!20!white)3 & *(green!20!white)3& *(red!30!white)2   &*(red!30!white)2  & *(red!30!white)1 & \\
& &  *(green!20!white) 1 &*(green!20!white) 2 & *(green!20!white)2& *(red!30!white)2  & *(red!30!white)1 &  &  \\
& &    & *(green!20!white) 1& *(green!20!white)2&*(red!30!white) 1   &  &  &  \\
& &    &  &*(green!20!white) 1&    &  & &  \\
\end{ytableau}
\hskip2cm
\begin{ytableau}
\textcolor{white}{.} &  & &  &  *(yellow!20!white) 1 &    &  & & \\
 & &    &  *(yellow!20!white) 1& *(yellow!20!white) 2& *(violet!40!white)  1 &  &   & \\
 &  &   *(yellow!20!white)1 &  *(yellow!20!white)2& *(yellow!20!white)2 & *(violet!40!white)2 &  *(violet!40!white)1&  & \\
&  *(yellow!20!white) 1&   *(yellow!20!white)2 & *(yellow!20!white)3 & *(yellow!20!white)3 &  *(violet!40!white)2 &  *(violet!40!white)2& *(violet!40!white)1 &  \\
*(green!20!white) 1& *(green!20!white) 2&   *(green!20!white)3 &  *(green!20!white)3& *(teal!85!white)  3&  *(violet!40!white) 3  & *(violet!40!white) 3  & *(violet!40!white) 2 & *(violet!40!white) 1\\
& *(green!20!white)1&   *(green!20!white)2 & *(green!20!white)2 & *(red!30!white)3& *(red!30!white)3   &*(red!30!white)2  & *(red!30!white)1 & \\
& &  *(green!20!white) 1 &*(green!20!white) 2 & *(red!30!white)2& *(red!30!white)2  & *(red!30!white)1 &  &  \\
& &    & *(green!20!white) 1& *(red!30!white)2&*(red!30!white) 1   &  &  &  \\
& &    &  &*(red!30!white) 1&    &  & &  \\
\end{ytableau}}
\]
\end{example}

\subsection{Regime 3: Termination outside of the Aztec diamond}\label{sec:thirdregimeaztermination}
Henceforth, let $N=\lceil \frac{n}{\sqrt{3}} \rceil +1$.

\begin{reptheorem}{thm:3regimes}[Part \ref{3regime3-b}]
If $r \geq N$, then  $\K(n,r)$ cannot terminate in the Aztec diamond $\az(n)$.
\end{reptheorem}

\begin{rem} In fact, Theorem \ref{thm:3regimes}(\ref{3regime3-b}) holds for any closed curve of weight $n$ containing $\K(N,r)$.
\end{rem}

\begin{proof}
It is sufficient to prove that $\weight(\K(n,r)) > \weight(\az(n))$.  Let $\K$ denote a single quadrant of the Aztec diamond. The weight of $\K$ is
\begin{align*} 
\weight (\K) =(n + n-1 + \cdots + 1) + (n-1 + n-2 + \cdots + 1) + \cdots + 1= \frac{n(n+1)(n+2)}{6}.
\end{align*}
Thus summing over all four quadrants \emph{and} $f_0$, we obtain
\[
\weight(\az(n)) = n+4 \left(\frac{n(n+1)(n+2)}{6} \right) = n+ \frac{2}{3}(n)(n+1)(n+2)\,.
\]
Let $\K'$ denote a single quadrant of  $\K(n,r)$. The weight of $\K'$ is
\[
\weight (\K')= nr + n(r-1) + \cdots 2n + n =  n\cdot \frac{r(r+1)}{2}
\]
where the first expression comes from adding the total number of chips in each row of the quadrant. Thus by similar logic, summing over all quadrants and $f_0$ gives
\[ \weight(\K(n,r)) = n + 4  n\cdot \frac{r(r+1)}{2}. \]
Suppose $r \geq \lceil \frac{n}{\sqrt{3}} \rceil +1$, and note that for $n$ a positive integer, 
$ \lceil n/\sqrt{3} \rceil > n/\sqrt{3}. $ Thus 
\begin{align*}
    \weight(\K(n,r)) & >  \frac{2}{3}\left( n^3 + 3\sqrt{3}n^2 + 6n \right) + n> \frac{2}{3} \left( n^3 + 3n^2 + 2n \right) + n = \weight(\az(n)).  \qedhere
\end{align*} 
\end{proof}

\subsection{Bound improvement}\label{sec:closing}
Theorem \ref{thm:3regimes}(\ref{3regime3-b}) naturally raises the question of whether it is possible for $\K(n,r)$ to terminate in the Aztec diamond when $ \lceil \frac{n}{\sqrt{3}} \rceil +1> r >\lceil n/2 \rceil$. In the table below, we present the values of $r =  \lceil \frac{n}{\sqrt{3}} \rceil +1$ and $ \lceil n/2 \rceil $ for $3 \leq n \leq 24$, along with the minimum value of $r$ such that $\weight (\K(n,r))> \weight (\az(n))$.

\begin{center}
\scalebox{0.85}{
\begin{tabular}{|c |c| c| c| c| c| c| c| c |c |c |c| c| c| c| c| c |c |c |c |c |c |c |c  |c||}
 \hline
\textbf{n}& \textbf{3}& \textbf{4} &\textbf{5} & \textbf{6}& \textbf{7}& \textbf{8} & \textbf{9} & \textbf{10} & \textbf{11}& \textbf{12} & \textbf{13} & \textbf{14} & \textbf{15} & \textbf{16} & \textbf{17} & \textbf{18} & \textbf{19} & \textbf{20} & \textbf{21} & \textbf{22} & \textbf{23} & \textbf{24}  \\
 \hline
$\mathbf{\lceil \frac{n}{2} \rceil}$& 2& 2 &3 & 3& 4& 4&5&5&6&6&7&7& 8 &8 &9 &9 &10& 10 &11 & 11 & 12 & 12  \\
 \hline
\textbf{Minimum r}& 2& 3 &4 & 4& 5& 6&6&7&7&8&8&9& 10 &10 &11 &11 & 12&  12&  13& 14&  14&  15\\
 \hline
$\mathbf{\lceil \frac{n}{\sqrt{3}} \rceil +1}$ & 3& 4&4 &5 & 6& 6&7&7&8&8&9&10&10  &11 &11 &12 &12&13 & 14& 14& 15& 15\\
 \hline
\end{tabular}}

\end{center}

The minimum value of $r$ for which $\weight(\K(n,r))> \weight(\az(n))$ is precisely  $\lceil \frac{n}{\sqrt{3}} \rceil +1$ approximately one third of the time. Thus, in these cases, the bound for Theorem \ref{thm:3regimes}(\ref{3regime3-b}) cannot be improved unless a  different approach is taken. On the other hand, for $n > 8$, there is a gap between $\lceil \frac{n}{2} \rceil$ and the minimum value of $r$. This suggests finding a better upper bound for Regime 2 may be more viable. Below are examples that show improving the bounds for Regime 2 and Regime 3 may be quite subtle. 

\begin{enumerate}
\item {\bf It is possible to find configurations $\K(n,r)$ with $ \lceil \frac{n}{2} \rceil < r < \lceil \frac{n}{\sqrt{3}} \rceil +1$ that can terminate in $\az(n)$.} Consider the case $n=8$; the only ball contained in the region between Regimes 2 and 3 is $\K(8,5)$. First observe that $ \weight(\K(8,5))=\weight(\az(8))$. For this initial configuration, there is a sequence of firing moves that terminate at $\az(8)$. 

\item {\bf It is possible to find configurations $\K(n,r)$ with $ \lceil \frac{n}{2} \rceil < r < \lceil \frac{n}{\sqrt{3}} \rceil +1$ that cannot terminate in $\az(n)$.} Consider the case $n=4$; the  only ball contained in the region between Regimes 2 and 3 is $\K(4,3)$. As one can see from the table above, $ \weight(\K(4,3))>\weight(\az(4))$. So, it is not possible for $\K(4,3)$ to terminate at $\az(4)$.

\item {\bf It is possible to find configurations within $\K(n,r)$ that can terminate in $\az(n)$, even if $\K(n,r)$  cannot.} Suppose that $n = 4$. Consider the configuration  on the left in the figure which is contained in $\K(4,3)$:

\[
\scalebox{0.8}{\begin{ytableau}
\textcolor{white}{.} &  & & *(blue!40!white)  4 &   & &      \\
 &  & *(blue!40!white)  4 & *(blue!40!white)  4& *(blue!40!white) 4 & &    \\
  & *(black!20!white) & *(blue!40!white) 4 & *(blue!40!white)  4 &*(blue!40!white)  4 & *(black!20!white) &    \\
   *(blue!40!white)  4 & *(blue!40!white)  4 & *(blue!40!white)  4 & *(teal!50!white) 4 & *(blue!40!white)  4 & *(blue!40!white)  4 & *(blue!40!white)  4    \\
  & *(black!20!white) &*(blue!40!white) 4 & *(blue!40!white) 4 &*(blue!40!white)  4 & *(black!20!white)&   \\
 & &*(blue!40!white)  4 &*(blue!40!white)  4 & *(blue!40!white)  4 & &   \\
 & & & *(blue!40!white) 4 & & &    \\
\end{ytableau}}
\hskip2cm
\scalebox{0.8}{
\begin{ytableau}
\textcolor{white}{.}  &  & & *(blue!40!white)  4 &   & &      \\
 &  & *(blue!40!white)  4 & *(blue!40!white)  4& *(blue!40!white) 4 & &    \\
 & *(blue!40!white) 4 & *(blue!40!white) 4 & *(blue!40!white)  4 &*(blue!40!white)  4 & *(blue!40!white) 4&   \\
   *(blue!40!white)  4 & *(blue!40!white)  4 & *(blue!40!white)  4 & *(teal!50!white) 4 & *(blue!40!white)  4 & *(blue!40!white)  4 & *(blue!40!white)  4    \\
  & *(black!20!white) &*(blue!40!white) 4 & *(blue!40!white) 4 &*(blue!40!white)  4 & *(black!20!white)&    \\
& &*(black!20!white)  &*(blue!40!white)  4 & *(black!20!white)  & &   \\
&& & *(blue!40!white) 4 & & &    \\
\end{ytableau}}
\]

\noindent Note that $\weight (\K) = \weight (\az(4))$. For this initial configuration $\K$, one can find a sequence of firing moves that terminate at the Aztec diamond. 

\item {\bf It is possible to find a configuration $\K$ with $\weight(\K) = \weight(\az(n))$ that cannot terminate in $\az(n)$.} Again, take $n=4$ and consider the configuration on the right in the figure above. This configuration \emph{cannot} terminate in $\az(4)$.
\end{enumerate}

\bibliographystyle{plain}
\bibliography{refs}
\end{document}